\documentclass[11pt]{article}
\usepackage[margin=1in]{geometry} 
\usepackage{amssymb,amsfonts,amsmath,bbm,mathrsfs,stmaryrd}
\usepackage{xcolor}
\usepackage{url}

\usepackage[colorlinks,
             linkcolor=black!75!red,
             citecolor=blue,
             pdftitle={Dedekind complete posets from sheaves on von Neumann algebras},
             pdfauthor={Alexandru Chirvasitu},
             pdfsubject={operator algebras},
             pdfkeywords={oa.operator algebras, ct.category theory},
             pdfproducer={pdfLaTeX},
             pdfpagemode=None,
             bookmarksopen=true
             bookmarksnumbered=true]{hyperref}

\usepackage{tikz}
\usetikzlibrary{arrows,decorations.pathreplacing,decorations.markings,shapes.geometric,through,fit,shapes.symbols,positioning}

\usepackage[amsmath,thmmarks,hyperref]{ntheorem}
\usepackage{cleveref}

\crefname{section}{Section}{Sections}
\crefformat{section}{#2Section~#1#3} 
\Crefformat{section}{#2Section~#1#3} 

\crefname{subsection}{\S}{\S\S}
\crefformat{subsection}{#2\S#1#3} 
\Crefformat{subsection}{#2\S#1#3} 

\theoremstyle{change}

\newtheorem{lemma}{Lemma}[section]
\newtheorem{proposition}[lemma]{Proposition}
\newtheorem{corollary}[lemma]{Corollary}
\newtheorem{theorem}[lemma]{Theorem}

\theoremstyle{nonumberplain}

\theoremstyle{change}
\theorembodyfont{\upshape}
\theoremsymbol{\ensuremath{\blacklozenge}}

\newtheorem{definition}[lemma]{Definition}
\newtheorem{example}[lemma]{Example}
\newtheorem{remark}[lemma]{Remark}

\crefname{definition}{definition}{definitions}
\crefformat{definition}{#2definition~#1#3} 
\Crefformat{definition}{#2Definition~#1#3} 

\crefname{lemma}{lemma}{lemmas}
\crefformat{lemma}{#2lemma~#1#3} 
\Crefformat{lemma}{#2Lemma~#1#3} 

\crefname{proposition}{proposition}{propositions}
\crefformat{proposition}{#2proposition~#1#3} 
\Crefformat{proposition}{#2Proposition~#1#3} 

\crefname{example}{example}{examples}
\crefformat{example}{#2example~#1#3} 
\Crefformat{example}{#2Example~#1#3} 

\crefname{remark}{remark}{remarks}
\crefformat{remark}{#2remark~#1#3} 
\Crefformat{remark}{#2Remark~#1#3} 

\crefname{corollary}{corollary}{corollaries}
\crefformat{corollary}{#2corollary~#1#3} 
\Crefformat{corollary}{#2Corollary~#1#3} 

\crefname{theorem}{theorem}{theorems}
\crefformat{theorem}{#2theorem~#1#3} 
\Crefformat{theorem}{#2Theorem~#1#3} 

\crefname{equation}{}{}
\crefformat{equation}{(#2#1#3)} 
\Crefformat{equation}{(#2#1#3)}

\theoremstyle{nonumberplain}
\theoremsymbol{\ensuremath{\blacksquare}}

\newtheorem{proof}{Proof}
\newtheorem{proof of main}{Proof of \Cref{th.main}}
\newtheorem{proof of main_bis}{Proof of \Cref{th.main_bis}}

\DeclareMathOperator{\id}{id}

\newcommand\bN{{\mathbb N}}

\newcommand\bR{{\mathbb R}}
\newcommand\bC{{\mathbb C}}

\newcommand\cB{{\mathcal B}}
\newcommand\cC{{\mathcal C}}
\newcommand\cD{{\mathcal D}}
\newcommand\cF{{\mathcal F}}

\newcommand\wt{\widetilde}
\newcommand\wh{\widehat}
\DeclareMathOperator{\Hom}{Hom}
\DeclareMathOperator{\2}{II_1}
\DeclareMathOperator{\dv}{(\cat{div})}

\newcommand{\define}[1]{{\em #1}}

\newcommand{\cat}[1]{\textsc{#1}}

\newcommand{\qedhere}{\mbox{}\hfill\ensuremath{\blacksquare}}

\title{Dedekind complete posets from sheaves on von Neumann algebras}
\author{Alexandru Chirvasitu\footnote{University of California at Berkeley, \url{chirvasitua@math.berkeley.edu}}}

\begin{document}
\maketitle

\begin{abstract}
We show that for any two von Neumann algebras $M$ and $N$, the space of non-unital normal homomorphisms $N\to M$ with finite support, modulo conjugation by unitaries in $M$, is Dedekind complete with respect to the partial order coming from the addition of homomorphisms with orthogonal ranges; this ties in with work by Brown and Capraro, where the corresponding objects are given Banach space structures under various niceness conditions on $M$ and $N$. More generally, we associate to $M$ a Grothendieck site-type category, and show that presheaves satisfying a sheaf-like condition on this category give rise to Dedekind complete lattices upon modding out unitary conjugation. Examples include the above-mentioned spaces of morphisms, as well as analogous spaces of completely positive or contractive maps. We also study conditions under which these posets can be endowed with cone structures extending their partial additions. 
\end{abstract}

\noindent {\em Keywords: von Neumann algebra, poset, lattice, Dedekind complete, sheaf, sketch, cone}

\tableofcontents

\section*{Introduction}

Organizing objects one is interested in into ``moduli spaces'' and then endowing such spaces with as much structure as possible is one of the more familiar themes in mathematics. As an instance of this, with the objects in question being operator algebra homomorphisms,  Brown studies in \cite{MR2799809} the space $\Hom(A,B)$ of normal unital homomorphisms between two $W^*$-algebras $A,B$ up to conjugation by unitaries in $B$ (and similarly a $C^*$ version). As it turns out, when $B$ is an ultrapower $R^\omega$ of the hyperfinite $\2$ factor and $A$ is a separable, embeddable $\2$ factor (so as to define away Connes' embedding problem, stating that \define{any} separable $\2$ factor embeds in such a manner), this space exhibits a rich structure indeed: It can be given a complete metric and a convex structure compatible with this metric in the sense of \cite[2.1]{MR2799809}. 

This analysis is continued in \cite{MR2997389}. The authors show that the similarly-defined moduli space of (this time non-unital) homomorphisms $N\to M\otimes\cB(\ell^2(\bN))$ with finite supporting projection can be organized naturally as a cancellable semigroup. Moreover, when $N$ is a separable $\2$ factor and $M=X^\omega$ for some McDuff factor $X$ (i.e. $X\cong X\otimes R$, for the spatial von Neumann algebra tensor product), this semigroup can be made into a cone, and its enveloping group then becomes a vector space. Furthermore, the same type of metric as in \cite{MR2799809} makes this vector space into a Banach space. 

The motivation for the present paper is the following question: Given that the above constructions provide (a) a Banach space and (b) a cone sitting naturally inside of it, is it the case that the resulting partial order makes the Banach space into a Banach lattice? Once one is led to ask this though, the problem quickly becomes dissociated into two more or less independent halves: There is the Banach space aspect, perhaps requiring technical tools like McDuff-ness, and there is the question of whether or not the order on $\Hom(N,M\otimes B(\ell^2(\bN)))$ coming from its semigroup structure (which makes sense under very general circumstances) is complete. The main goal here is to answer this latter question affirmatively, making the connection with the former problem only towards the end. 

As we explain in \Cref{se.complete} (at the very end, after \Cref{prop.wedge_vee}), the results on completeness may be regarded as providing some small measure of evidence that the Banach spaces introduced in \cite{MR2997389} have preduals. This, in turn, would answer affirmatively a question by Sorin Popa, to the effect that embeddable factors have embeddings in $R^\omega$ with factorial commutant. Secondly, the same completeness results have consequences on the construction of the vector space structure itself, simplifying it somewhat; this is taken up in \Cref{se.cone}.

The contents of the paper are as follows:

In an attempt to isolate the main abstract features of the problem, in \Cref{se.sketch} we set up a framework where invariants of a von Neumann algebra $M$ such as $\Hom(N,M)$ mentioned above or similar moduli spaces of completely positive or completely bounded maps $N\to M$ can be treated simultaneously. This will require associating to $M$ a category $\cC_M$ with distinguished cocones (a structure referred to as a `sketch'; see below), and recasting these invariants as ``sheaves'' on this sketch. This is spelled out and clarified in \Cref{se.sketch}, but sheaves are essentially presheaves on $\cC_M$ which ``respect the cocone structure''. This will be more of a preliminary section, aimed at fixing the general setup; the only result we actually state here is \Cref{prop.monoid}, and for the proof we refer to that of \cite[9]{MR2997389}, which can easily be paraphrased into the present setting. 

\Cref{se.complete} contains the first main result of the paper, \Cref{th.main}. It is stated in the general language of sheaves on sketches, but the connection to the previous discussion is as follows: For any two $W^*$-algebras $M,N$ whatsoever, the set of non-unital morphisms $N\to M$ with finite supporting projection modulo conjugation by unitaries in $M$ is a Dedekind complete poset with respect to the ordering induced by its partial abelian monoid structure (\Cref{prop.monoid}). This means that any subset, regardless of cardinality, has a largest lower (least upper) bound as soon as it is bounded from below (respectively above).   

In \Cref{se.cone} we move to the ``continuous'' side of the problem, addressing the issue of when the monoids from the preceding section can be endowed with something like a vector space structure. The precise notion we need is that of partial cone structure, i.e. action of the partial semiring $([0,1],+,\cdot)$, and the relevant result is \Cref{th.main_bis}. It follows, only from the results on the completeness of the ordering, that such cone structures are unique when they exist. In conclusion, not much care is required when constructing one. 

Finally, \Cref{se.adjoints} contains only announcements of results to be proven elsewhere. Since the somewhat abstract categorical machinery of \Cref{se.sketch} was partly justified by the natural desire to study completely positive or contractive maps along with $W^*$ morphisms, it seems fair to point out that in fact, these do not give rise to fundamentally new invariants. Specifically (\Cref{prop.adjoints}), for von Neumann algebras $M,N$, the set of completely positive maps $N\to M$ is (functorially in $M$ and $N$) in bijection with $W^*$ morphisms $\wt N\to M$ for some von Neumann algebra $\wt N$; similarly for completely contractive maps.   

The point of \Cref{se.sketch}, however, is precisely that there is an abstract framework that will instantly accomodate all of these invariants (and similar ones that might come up); one need not go through results in the spirit of \Cref{prop.adjoints}, which are maybe a bit surprising and not immediately obvious.

\subsection*{Acknowledgements}

I would like to thank Valerio Capraro for bringing many of these problems to my attention (my interest was sparked following a talk he gave at UC Berkeley in the spring of 2013) and for his patient explanations on the contents \cite{MR2997389}. Thanks to Andre Kornell are also in order, for numerous discussions on the results covered in \Cref{se.adjoints}.

\section{The canonical sketch and sheaves on it}\label{se.sketch}

Throughout, capital letters such as $M,N$, etc. denote $W^*$-algebras. We will be taking for granted the basics of the theory of von Neumann algebras (a term used here as a synonym for `$W^*$-algebra'), as in \cite{MR1873025}, say. There are portions of the paper (e.g. the end of \Cref{se.complete}), offered mostly as motivation, where some of the main characters relating to Connes' embedding problem enter the discussion (the hyperfinite $\2$ factor, ultrafilters, etc.). For these I recommend the preliminary section of \cite{2010arXiv1003.2076C}, where all of the relevant notions are gathered together and presented very accessibly and to the point. 

All maps between von Neumann algebras are assumed to be weak$^*$ continuous (or normal), an assumption that will always be implicit, even when not repeated explicitly. $W^*$ morphisms or homomorphisms are $*$-preserving, multiplicative maps. They may or may not be unital, depending on context; it will always be clear which case we are considering. 

Whenever the symbol `$\otimes$' appears between two $W^*$-algebras, it denotes spatial tensor products (e.g. \cite[IV.5]{MR1873025}). The matrix algebra $M_n=M_n(\bC)$ is given the usual $*$-structure, via the transpose conjugate. 

We mention completely positive and completely contractive maps between $W^*$-algebras only very briefly, in \Cref{se.complete} and \Cref{se.adjoints}. Recall that a map $\theta:N\to M$ between $W^*$-algebras is a \define{complete contraction} \cite[p. 167]{MR1943007} if the maps $\theta\otimes\id:N\otimes M_n(\bC)\to M\otimes M_n(\bC)$ all have norm $\le 1$. Similarly, $\theta$ is \define{completely positive} \cite[IV.3]{MR1873025} if $\theta\otimes\id$ are all positive, in the sense that they turn positive elements $x^*x\in N\otimes M_n$ into positive elements $y^*y\in M\otimes M_n$. Just like $W^*$ morphisms, completely positive or contractive maps may or may not be unital (i.e. send $1$ to $1$), and we will specify or otherwise make it clear when they do and when they do not. 

We also assume some of the standard background on category theory, as covered by \cite{MR1712872}; this means functors, natural transformations, (co)limits, adjoint functors etc. In addition, \cite{MR1294136} is a natural source for some of the concepts we are about to review. We write $\cC(c,c')$ for the set of morphisms $c\to c'$ in a category $\cC$. 

Let $\cC$ be a category, and $\cD\subseteq \cC$ a subcategory. A \define{cocone} on $\cD$ in $\cC$ (simply called a cone in \cite[III.3]{MR1712872}) consists of an object $\bullet\in\cC$ and a natural transformation from the inclusion $\cD\subseteq\cC$ to the functor $\cD\to\cC$ constantly equal to $\bullet$. Explicitly, this means each object $d\in\cD$ is endowed with an arrow $d\to\bullet$ in $\cC$, and for each arrow $d\to d'$ in $\cD$ the resulting triangle 
	\[
		\tikz[anchor=base]{
  		\path (0,0) node (1) {$d$} +(2,0) node (2) {$d'$} +(1,-.5) node (3) {$\bullet$};
  		\draw[->] (1) -- (2);
  		\draw[->] (2) -- (3);
  		\draw[->] (1) -- (3); 
 		}
	\]
is commutative. Cones are defined dually, and consist of arrows $\bullet\to d$, etc. Finally, a cone is \define{limiting} if it makes $\bullet$ into the limit of the functor $\cD\subseteq\cC$ (\cite[III.4]{MR1712872}).

\begin{definition}\label{def.sketch}
A \define{sketch} consists of a small category $\cC$ together with a collection of distinguished cocones. We typically abuse terminology and refer to $\cC$ alone as being the sketch, with the cocones being implicit. 

A \define{sheaf} on a sketch $\cC$ is a contravariant functor $\cC\to\cat{Set}$ which turns the distinguished cocones into limiting cones. 
\end{definition}

Because the data on $\cC$ consists of cocones only (as opposed to both cones and cocones), our sketches are referred to as \define{colimit} sketches in \cite[2.55]{MR1294136}. The current usage is consistent with that in \cite{raey}, where colimit sketches are similarly the main characters.

\begin{example}
Some of the more ubiquitous examples of sketches are the so-called sites, also known as Grothendieck topologies \cite[III]{MR1300636}. These are small categories whose objects and arrows are supposed to be thought of as the open sets of some kind of ``space'', and inclusions between them, respectively. By definition, they come endowed with distinguished cocones, which heuristically indicate which collections of open sets ``cover'' other open sets. Finally, what's typically referred to as a sheaf on a site \cite[III.4]{MR1300636} is precisely a sheaf in the sense of \Cref{def.sketch}. 

This is the example that motivated the term `sheaf' in \cite[2.2.1]{raey} on a sketch, even in cases (such as that of \Cref{def.can_sketch} below) when the sketch is not necessarily a site.    
\end{example}

\begin{definition}\label{def.can_cat}
Let $M$ be a von Neumann algebra. The \define{finite subalgebra category} $\cC_M$ of $M$ has as objects the not-necessarily-unital $W^*$-subalgebras of $M$ with finite supporting projection. A morphism $A\to B$ between such objects is a partial isometry $u\in M$ such that:
	\begin{itemize}
		\item the initial space $u^*u$ of $u$ is the supporting projection $1_A$ of $A$; 
		\item the final space $uu^*$ is dominated by the supporting projection $1_B$ of $B$ (i.e. $uu^*\le 1_B$ in the usual partial ordering on projections);
		\item $u$ conjugates $A$ into $B$, in the sense that $uAu^*\le B$. 
	\end{itemize} 
Composition of morphisms is defined in the obvious manner, as multiplication of the underlying partial isometries. 
\end{definition}

\begin{remark}
\Cref{def.can_cat} works just fine for arbitrary subalgebras, with perhaps infinite supporting projection, but finiteness is a technical tool need below (e.g. it is crucial to the cancellability results \cite[9]{MR2997389} and \Cref{prop.monoid}; see also \Cref{ex.mod_space}). 

We sometimes drop the adjective `finite' and just refer to $\cC_M$ as the subalgebra category. 
\end{remark}

Finally, we come to the main definition of this section:

\begin{definition}\label{def.can_sketch}
The \define{canonical sketch} on the subalgebra category $\cC_M$ of $M$ consists of the following cocones: 
\renewcommand{\labelenumi}{(\alph{enumi})}
\begin{enumerate}
\item pushout-type ones of the form
	
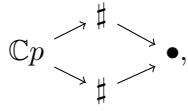
\[
		\tikz[anchor=base]{
  		\path (0,0) node (1) {$\bC p$} +(1,.5) node (2) {$\sharp$} +(1,-.5) node (3) {$\sharp$} +(2,0) node (4) {$\bullet$,};
  		\draw[->] (1) -- (2);
  		\draw[->] (1) -- (3);
  		\draw[->] (2) -- (4);
  		\draw[->] (3) -- (4); 
 		}
	\]
where $p\in M$ is a projection, the $\sharp$'s are $\le 2$-dimensional algebras generated by subprojections $q\le p$ and $p-q$ (possibly different $q$'s for the two $\sharp$'s), $\bullet$ is the subalgebra generated by the two $\sharp$'s, and the arrows are inclusions, i.e. as partial isometries, they are all equal to $p$.; 
\item coproduct-type ones, consisting of the inclusions $A_i\to\bigoplus A_i$ into the direct sum of a family of mutually orthogonal subalgebras $A_i\le M$, $i\in I$. $I$ can be any set whatsoever, including empty; the direct sum is then understood to be the zero subalgebra of $M$.  
\end{enumerate}
\end{definition}

Recall that sketches are supposed to be cones in $\cC_M$. Hence, it is implicit in \Cref{def.can_sketch} that all supporting projections in sight (such as that of $\bigoplus A_i$) are finite in $M$.

\begin{remark}
\Cref{def.can_sketch} follows a familiar pattern when dealing with (co)limits of functors: All colimits can be built out of just pushouts and arbitrary coproducts (and dually, limits can be constructd using pullbacks and products; this is a variant of \cite[V.2]{MR1712872}), so cocones of the types (a) and (b) are what one might expect to work with. 

Note that when referring to (a) and (b) as `pushout-type' and `coproduct-type', the terms are only intended to suggest analogies; these cocones are not, in general, genuine colimits in $\cC_M$. The sheaf condition (\Cref{def.sketch}) will demand that they be turned into limits in $\cat{Set}$ however.  
\end{remark}

\begin{definition}\label{def.sheaf}
For a $W^*$-algebra $M$, a \define{sheaf on $M$} is a sheaf in the sense of \Cref{def.sketch} on the sketch $\cC_M$. 
\end{definition}

\begin{remark}\label{rem.empty_prod}
In \Cref{def.can_sketch}, we allowed the indexing set $I$ of a coproduct-type cocone to be empty. The corresponding condition on a sheaf is that it turn the zero algebra into the terminal object in $\cat{Set}$, i.e. the singleton. 
\end{remark}

There are a few examples of sheaves on $M$ that the reader should have in mind, and which provided the motivation for these definitions to begin with:

\begin{example}\label{ex.Wast_mor}
Let $N$ be another $W^*$-algebra, and consider the sheaf $\cat{W}^*_N$ sending a finitely-supported subalgebra $A\le M$ to the set of unital $W^*$ morphisms $N\to 1_AM1_A$ whose image commutes with $A$ (we will simply say that the maps themselves commute with $A$). 

This is easily seen to define a persheaf on $\cC_M$ (i.e. a contravariant functor into $\cat{Set}$) essentially because whatever commutes with an algebra also commtues with a smaller algebra. We leave it to the reader to check that the presheaf turns the cocones of \Cref{def.sketch} into limits in $\cat{Set}$.  
\end{example}

\begin{example}\label{ex.cp}
We can play the exact same game with completely positive maps rather than $W^*$ morphisms. Indeed, for $N$ as before, we can define the sheaf $\cat{CP}_N$ ($\cat{CP}$ for completely positive) sending $A\in\cC_M$ to the set of unital completely positive maps $N\to 1_AM1_A$ commuting with $A$ 
\end{example}

\begin{example}\label{ex.cc}
Repeat previous example with completely contractive maps to get a sheaf $\cat{CC}_N$. 
\end{example}

\begin{remark}\label{rem.stack}
If $\cC_M$ is heuristically a site, we ought to be able to make sense of more sophisticated notions usually associated to sites, such as, say, stacks. These would be ``categorified'' sheaves, in the sense of being contravariant functors from $\cC_M$ to the bicategory of categories, functors, and natural transformations, with an additional compatibility with the cocones (e.g. \cite{MR1905329}). 

Associating to $A\in\cC_M$ the category of $A-1_AM1_A$ Hilbert space bimodules (i.e. correspondences from $1_AM1_A$ to $A$ in the sense of Connes' \cite[Chapter 5, Appendix B]{MR1303779}) produces precisely such a stack-like object. We will not elaborate, as these ideas play no further role in the paper.  
\end{remark}

Finally, we define objects analogous to the morphisms-up-to-conjugation sets of \cite{MR2799809,MR2997389}.

\begin{definition}\label{def.mod_pseudo}
A morphism $u:A\to B$ in $\cC_M$ is a \define{pseudoisomorphism} if the final space $uu^*$ of the underlying partial isometry $u$ is precisely $1_B$ (as opposed to just being dominated by $1_B$; cf. \Cref{def.can_cat}). 

Let $\cF$ be a sheaf on $M$. The \define{moduli space of $\cF$-objects} $\pi(\cF)$ is a set defined as follows:

The elements are equivalence classes of elements in the disjoint union \[ \coprod_{A\in\cC_M} \cF(A), \] where $x\in\cF(A)$ and $y\in\cF(B)$ are equivalent if there is some pseudoisomorphism $u:A\to B$ in $\cC_M$ such that $x=\cF(u)(y)$.

We write either $[x]$ or just plain $x$ for the class of $x\in\cF(A)$ in $\pi(\cF)$. 
\end{definition}

\begin{example}\label{ex.mod_space}
Applying $\pi$ to the sheaf $\cat{W}^*_N$ of \Cref{ex.Wast_mor} produces exactly the set of non-unital morphisms $N\to M$ up to conjugation by unitaries in $M$. This is because even though in \Cref{def.mod_pseudo} we are conjugating by partial isometries, an equivalence by partial isometries between \define{finite} projections can always be extended to a unitary equivalence \cite[V.1.38]{MR1873025}. 
\end{example}

So far $\pi(\cF)$ is only a set, but it is naturally endowed with additional structure. 

Suppose first that $A,B\in\cC_M$ are such that the projections $1_A$ and $1_B$ are orthogonal (we simply say that $A$ and $B$ themselves are orthogonal). Then, by the sheaf conditions, $\cF(A\oplus B)$ is the product of the sets $\cF(A)$ and $\cF(B)$. Hence, for every $x\in\cF(A)$, $y\in\cF(B)$, it makes sense to define $x+y\in\cF(A\oplus B)$ as the unique element projecting on $x$ and $y$ through the canonical maps.  

More generall, if $u1_Bu^*$ is orthogonal to $1_A$ for some partial isometry $u$, $u^*u=1_B$, then one can still make sense of $x+y$ at the level of $\pi(\cF)$ by defining it as above in $\cF(A\oplus uBu^*)$. The definition, it turns out, does not depend on $u$ as long as $B$ is moved orthogonally to $A$ (cf. \cite[Definition 4]{MR2997389}). We record all of this in the next proposition. The only perhaps non-obvious claim is the cancellability of the partial addition; we do not repeat the argument from \cite[Proposition 9]{MR2997389}, which contains all of the needed ingredients and can easily be adapted here.

First, one more piece of terminology. This is a very guessable definition, but in any event:

\begin{definition}\label{def.part_mon}
Let $G$ be a set endowed with a partial binary operation `$+$' and an element $0\in G$. We say that $(G,+,0)$ is a \define{partial abelian monoid} if
	\begin{itemize}
		\item Whenever $x+y$ is defined so is $y+x$, and the two are equal;
		\item Whenever one of $(x+y)+z$ and $x+(y+z)$ is defined so is the other, and they are equal;
		\item $0+x=x$ for all $x\in G$ (in particular, $0+x$ is always defined). 
	\end{itemize}

A partial monoid $G$ as above (dropping the adjective `abelian' and abusing notation by omitting `$+$' and `$0$') is \define{cancellable} or \define{has cancellation} if $x+y=x+z$ implies $y=z$.  
\end{definition}

Recall from \Cref{rem.empty_prod} that for any sheaf $\cF$, the image of the zero algebra $\{0\}$ is a singleton.

\begin{proposition}\label{prop.monoid}
For any sheaf $\cF$ on $M$, the partial operation `$+$' defined above on $\pi(\cF)$ makes the latter into a partial abelian monoid, with $0$ being the (class in $\pi(\cF)$ of the) single element of $\cF(\{0\})$. Moreover, $\pi(\cF)$ is cancellable. \qedhere
\end{proposition}

Cancellation, in particular, will play a crucial role in the sequel.

\begin{remark}\label{rem.proj_monoid}
The supporting projection of an element $x\in\pi(\cF)$ is well defined up to unitary equivalence, and associating it to $x$ consistutes a morphism (in the obvious sense) from the partial monoid $\pi(\cF)$ to that of equivalence classes of finite projections in $M$. 
\end{remark}

\begin{remark}\label{rem.inf_sums}
Note that in fact there is even more structure on $\pi(\cF)$ than \Cref{prop.monoid} lets on: It makes sense to add \define{arbitrary} families $x_i\in\pi(\cF)$, so long as their supporting projections can be moved inside $M$ so that they form an orthogonal family.

We will also talk freely about possibly infinite sums of elements $x_i\in\cF(A_i)$ (so at the level of $\cF$ rather than $\pi(\cF)$) when the algebras $A_i$ are mutually orthogonal. 
\end{remark}

\section{Sheaves give rise to Dedekind complete posets}\label{se.complete}

We fix a von Neumann algebra $M$ and a sheaf $\cF$ on it for the remainder of the paper, and freely resort to the notation introduced in the previous section, including `$+$' for the partial addition on $\pi(\cF)$. For projections $p\in M$, we write $\cF(p)$ as short for $\cF(\bC p)$.

As mentioned at the beginning, the paper is mostly about the partial order $(\pi(\cF),\le)$ resulting from the monoid structure:

\begin{definition}\label{def.le}
For $x,y\in \pi(\cF)$, we write $x\le y$ if there is some $z\in\pi(\cF)$ such that $x+z=y$. 
\end{definition}

A priori, `$\le$' is only a binary relation. As expected however, we have:

\begin{lemma}\label{lem.le}
$(\pi(\cF),\le)$ is a poset. 
\end{lemma}
\begin{proof}
Reflexivity and transitivity are clear (the former because of the existence of an additive zero for `$+$', and the former by associativity). We only need to argue that `$\le$' is antisymmetric. 

Suppose $x\le y\le x$. By definition, we have $y=x+z$ and $x=y+w=x+z+w$. By cancellation we get $z+w=0$, and hence (as can easily be seen just from the definition of $\pi(\cF)$) $z=w=0$. 
\end{proof}

We now start focusing on the completeness properties of the poset $(\pi(\cF),\le)$ that are the main topic of the paper. First:

\begin{definition}\label{def.ded}
A poset $(P,\le)$ (or its underlying order $\le$) is \define{Dedekind complete} if every non-empty subset bounded from below has a largest lower bound and dually, every non-empty subset bounded from above has a least upper bound. 

We write $\bigwedge S$ and $\bigvee S$ for the largest lower bound and the least upper bound of a set $S$ respectively. 
\end{definition}

\begin{remark}
This is a somewhat non-standard definition: Dedekind completeness is usually considered only in the context of lattices, when pairs of elements are assumed to always have greatest lower and least upper bounds. This is not implicit in \Cref{def.ded}, and for good reason. 

As we will see later (\Cref{th.main}), $(\pi(\cF),\le)$ always has largest lower bounds for arbitrary subsets, but not least upper bounds, even for pairs of elements. This is why the relevant notion here is that of \Cref{def.ded}. 
\end{remark}

\begin{theorem}\label{th.main}
Keeping the notations fixed above, $(\pi(\cF),\le)$ is Dedekind complete. 
\end{theorem}

The proof consists of a series of lemmas, which we then assemble together to get the conclusion. First, we tackle the issue of greatest lower bounds. The next result gets us only partly there.

\begin{lemma}\label{lem.max_lower}
Let $x_i\in\cF(p_i)$, $i\in I$ be an arbitary non-empty family for projections $p_i\in M$. The set of lower bounds for the classes $[x_i]\in\pi(\cF)$ has maximal elements. 
\end{lemma}

\begin{remark}
The assertion is not that a greatest lower bound exists, which is the ultimate goal, but rather that maximal elements exist. These are elements not dominated strictly by any other element, but they themselves need not dominate everyone else, a priori. 
\end{remark}

\begin{proof}
We construct such a maximal element by a transfinite recursive process. The steps are indexed by ordinal numbers $\alpha$, and the $\alpha$'th step produces ever-increasing families $p_i^\beta$, $\beta<\alpha$ of projections and elements $x_i^\beta\in\cF(p_i^\beta)$ such that
	\begin{itemize}
		\item For each $i\in I$, the projections $p_i^\beta\le p$, $\beta<\alpha$ are mutually orthogonal;
		\item The $x_i^\beta$'s are summands of the $x_i$'s in the sense of \Cref{rem.inf_sums};
		\item For every ordinal number $\beta<\alpha$ and every pair $i,j\in I$ of indices, there are isomorphisms $u_{ji}^\beta: \bC p_i^\beta\cong \bC p_j^\beta$ in $\cC_M$ ($u_{ji}^\beta$ being partial isometries); 
		\item The isomorphisms $u_{ji}^\beta$ intertwine the objects $x_i^\beta$, in the sense that $x_i^\beta$ is equal to $\cF(u_{ji}^\beta)(x_j^\beta)$. 
	\end{itemize}

	\[
		\tikz[anchor=base]{
  		\path (0,0) node {$p_i$} +(2,0) node {$p_j$} +(0,1) node [minimum height=1cm] (1) [draw] {$p_i^0$} +(2,1) node [minimum height=1cm] (2) [draw] {$p_j^0$} +(0,2) node [minimum height=1cm] (3) [draw] {$p_i^1$} +(2,2) node [minimum height=1cm] (4) [draw] {$p_j^1$} +(0,3) node {$\vdots$} +(2,3) node {$\vdots$} +(0,4) node [minimum height=1cm] (5) [draw] {$p_i^\alpha$} +(2,4) node [minimum height=1cm] (6) [draw] {$p_j^\alpha$} +(0,5) node {$\vdots$} +(2,5) node {$\vdots$} +(-1,1) node [minimum height=1cm] {$x_i^0$} +(3,1) node [minimum height=1cm] {$x_j^0$} +(-1,2) node [minimum height=1cm] {$x_i^1$} +(3,2) node [minimum height=1cm] {$x_j^1$} +(-1,4) node [minimum height=1cm] {$x_i^\alpha$} +(3,4) node [minimum height=1cm] {$x_j^\alpha$};  	
  		\draw[->] (1) -- (2) node[pos=.5,auto] {$\scriptstyle u_{ji}^0$};
  		\draw[->] (3) -- (4) node[pos=.5,auto] {$\scriptstyle u_{ji}^1$};
  		\draw[->] (5) -- (6) node[pos=.5,auto] {$\scriptstyle u_{ji}^\alpha$};	
 		}
 	\]
In step $\alpha=0$ there is nothing to do: We already have an empty family of projections. Next, we need to describe what to do in step $\alpha$, having completed all previous steps.

{\bf Case 1:} Assume first that $\alpha$ is a successor ordinal, with predecessor $\alpha_-$. We have the families $p_i^\beta$, $x_i^\beta$ for $\beta<\alpha_-$ from the previous inductive step. By mutual orthogonality, we can make sense, for each $i$, of the infinite sum $y_i$ of all $x_i^\beta$'s (see \Cref{rem.inf_sums}). It is an element of $\cF(q_i)$ for $q_i=\bigoplus_\beta A_i^\beta$. 

Now denote $r_i=p_i-q_i$. By the sheaf condition demanding that $\cF$ turn direct sums of orthogonal subalgebras into direct products, we can find $z_i\in\cF(r_i)$ such that $x_i=y_i+z_i$. If the set $[z_i]$, $i\in I$ has no non-zero lower bounds, terminate the induction. If there is some non-zero lower bound, then we would be able to find non-zero subprojections $r'_i\le r_i$, elements $w_i\in\cF(r'_i)$, and isometries $u_{ji}^{\alpha_-}$ in $\cC_M$ between the $r'_i$'s intertwining the $w_i$'s such that the common class $[w_i]$ is a lower bound for all $z_i$. The prescription for the recursion is to take $x_i^{\alpha_-}=w_i$ and $p_i^{\alpha_-}=r'_i$.   

{\bf Case 2:} Now suppose $\alpha$ is a limit ordinal. Since the previous steps $\beta<\alpha$ already provide a family of subalgebras indexed as desired (by all ordinals less than $\alpha$), simply do nothing and move on to the next ordinal number.

Since successor ordinals increase the family of mutually orthogonal non-zero subalgebras strictly, the process must stop at some point. When it does denote \[ \cF\left(\bigoplus_\beta p_i^\beta\right)\ni y_i=\sum_\alpha x_i^\alpha. \] The resulting elements $[y_i]\in\pi(\cF)$ coincide because the $y_i$'s are intertwined by the $\cC_M$-isomorphisms $\bigoplus_\beta u_{ji}^\beta$. Moreover, this common $[y_i]$ is by construction majorized by all $[x_i]$'s and is maximal with this property (or else we could continue the recursion). 
\end{proof}

We are now ready to complete the greatest lower bound half of \Cref{th.main}.

\begin{lemma}\label{lem.greatest_lower}
For any family $x_i\in\cF(p_i)$, the set $\{[x_i]\}\subseteq \pi(\cF)$ has a greatest lower bound. 
\end{lemma}
\begin{proof}
Let $[y]$ be a maximal lower bound for the $[x_i]$'s, provided by \Cref{lem.max_lower}. We can think of $[y]$ as the common class $[y_i]$ for some $y_i\in\cF(q_i)$, $q_i\le p_i$. Letting $[z]=[z_i]\le [x_i]$ be the common class of $z_i\in\cF(r_i)$ for some subprojections $r_i\le p_i$, we have to show that $[z]\le [y]$ in $\pi(\cF)$. 

Cancellability allows us to subtract a maximal lower bound of $[y]$ and $[z]$ from all $[x_i]$, and hence to assume without loss of generality that $[y]\wedge [z]=0$. Under this hypothesis, the new goal is to show that $[z]=0$. 

Now fix $i\in I$. We will be working with $x_i$, $y_i$, $p_i$, etc., but drop the subscript to streamline the notation.

The fact that $[x]$ has both $[y]$ and $[z]$ as summands means that $x\in\cF(p)$ is in the image of both $\cF(\bC q\oplus\bC(p-q))\to\cF(p)$ and $\cF(\bC r\oplus\bC(p-r))\to\cF(p)$ (the maps associated by $\cF$ to the inclusions of $\bC p$ into the two $\le 2$-dimensional algebras $A_q$ and $A_r$ obtained by cutting $p$ by $q$ and $r$ respectively). The pushout-type sheaf condition (a) of \Cref{def.can_sketch} says that the diagram
	\[
		\tikz[anchor=base]{
  		\path (0,0) node (1) {$\cF(p)$} +(2,1) node (2) {$\cF(A_q)$} +(2,-1) node (3) {$\cF(A_r)$} +(4,0) node (4) {$\cF(A)$};
  		\draw[<-] (1) -- (2);
  		\draw[<-] (1) -- (3);
  		\draw[<-] (2) -- (4);
  		\draw[<-] (3) -- (4); 
 		}
	\]
is a pullback in $\cat{Set}$, where $A$ is the algebra generated by $p$, $q$ and $r$. By a slight abuse of notation, we write $x-y\in\cF(p-q)$ and $x-z\in\cF(p-r)$ for the summands of $x$ complementary to $y$ and $z$ respectively. The pairs \[ (y,x-y)\in\cF(A_q)\cong \cF(q)\times\cF(p-q) \] and \[ (z,x-z)\in\cF(A_r)\cong \cF(r)\times\cF(p-r) \] both get mapped to $x\in\cF(p)$ through the left hand side of the above diagram, and so, by the pullback property, there is a unique $w\in\cF(A)$ mapped to $(y,x-y)$ and $(z,x-z)$ by the right hand side of the same diagram.   

Observe that $q$ and $r$ intersect trivially. Indeed, the intersection is a central projection in $A$, and the summand of $w$ corresponding to this central projection is a common summand of $y$ and $z$. The claim follows from our assumption $[y]\wedge [z]=0$.

Now, since $q,r$ are projections with trivial intersection in a von Neumann algebra $A$, there is a partial isometry $u\in A$ moving $r$ orthogonally to $q$: Simply take $u$ to be the partial isometry component in the polar decomposition of $(p-q)r$ (this is a familiar trick; see e.g. \cite[V.1.6]{MR1873025}). Since $u$ implements an automorphism of $A\in\cC_M$, it doesn't change the class $[w]=[x]$. But then it follows that $[y]+[z]\le [x]$, and so, by the maximality of $[y]=[z]=0$.  
\end{proof}

We are now ready to prove the main result of this section.

\begin{proof of main}
\Cref{lem.greatest_lower} takes care of half the theorem. The result still left to prove is that for an arbitrary family $x_i\in\cF(p_i)$, $i\in I$ such that $[x_i]$ have a common upper bound $[x]$, $x\in\cF(p)$, the set $\{[x_i]\}$ has a least upper bound. 

By \Cref{lem.greatest_lower}, the $[x]-[x_i]$'s (which make unambiguous sense by the cancellability claim in \Cref{prop.monoid}) have a greatest lower bound, say $[x]-[y]$ for some $[y]\le[x]$. The claim now is that $[y]$ is a least upper bound for the $[x_i]$'s. In other words, we want to show \[ [z]\ge [x_i],\ \forall i\Rightarrow [z]\ge [y]. \] 

Substituting $[z]\wedge[x]$ for $[z]$ (the former exists by \Cref{lem.greatest_lower} again), we may as well assume $[z]\le [x]$. But then $[z]\ge [x_i]$ is equivalent to $[x]-[z]\le [x]-[x_i]$, which in turn, by the greatest lower bound property of $[x]-[y]$, is equivalent to $[x]-[z]\le [x]-[y]$. In turn, this gets us the desired result $[z]\ge[y]$.
\end{proof of main}

Finally, the following observation is needed in the next section.

\begin{proposition}\label{prop.wedge_vee}
For any sheaf $\cF$ on a von Neumann algebra $M$ and any $y,z\in\pi(\cF)$ with a common upper bound, we have 							\begin{equation}\label{eq.wedge_vee} 
		y\vee z - y = z-y\wedge z. 
	\end{equation} 
\end{proposition}

As usual, subtraction makes sense by cancellation. The greatest lower bound always exists by \Cref{lem.greatest_lower}, while the least upper bound exists by \Cref{th.main} and the assumption that $y$ and $z$ have at least one common upper bound.

\begin{proof}
First, we reduce to the case $y\wedge z=0$.

Note that $y\vee z-y\wedge z$ is greater than both $y-y\wedge z$ and $z-y\wedge z$, and hence \[ y\vee z-y\wedge z\ge (y-y\wedge z)\vee(z-y\wedge z). \] in particular, the sum 
	\begin{equation}\label{eq.long_wedge} 
		(y-y\wedge z)\vee(z-y\wedge z) + y\wedge z 
	\end{equation}
is defined, and it is at most $y\vee z$. On the other hand, \Cref{eq.long_wedge} is clearly greater than (or equal to) both $y$ and $z$, and so 
	\begin{equation}\label{eq.long_wedge2} 
		(y-y\wedge z)\vee(z-y\wedge z) = y\vee z - y\wedge z. 
	\end{equation} 
Now, denoting $y-y\wedge z$ by $y'$ and $z-y\wedge z$ by $z'$, \Cref{eq.wedge_vee} becomes (via \Cref{eq.long_wedge2}) $y'\vee z' = y'+z'$, under the assumption $y'\wedge z'=0$. We henceforth revert to $y$ and $z$, make the additional assumption $y\wedge z=0$ (as announced at the beginning of the proof), and seek to show $y\vee z=y+z$. 

The proof of the latter claim is essentially contained in that of \Cref{lem.greatest_lower}. Realize $y$, $z$ and some common upper bound ($y\vee z$, say) as classes of elements in $\cF(q)$, $\cF(r)$ and $\cF(p)$ respectively, where $p$ is a projection in $M$ and $q,r\le p$ are sub-projections. Then, as in the proof of the lemma, $r$ can be moved orthogonally to $q$ by some partial isometry in the algebra $A$ generated by $q$, $r$ and $p$. The conclusion, as in the last line of the proof of the lemma, is $y+z\le y\vee z$.   
\end{proof}

We end this section by pointing to one possible interpretation of \Cref{th.main}.

Specialize for the remainder of the section to the main case studied in \cite{MR2799809}: $M$ is an ultrapower $R^\omega$ of the hyperfinite $\2$ factor $R$ (see \cite[$\S$2]{2010arXiv1003.2076C}) and $N$ is a separable $\2$ factor. The invariant studied by Brown is a close relative of $\pi(\cat{W}^*_N)$; it is defined in precisely the same way, out of only \define{unital} morphisms $N\to M$. For this reason, we denote it by $\pi_u$. 

Although a priori it might not be clear what kind of structure one should expect to have on $\pi_u$, \cite[4.6]{MR2799809} makes it into a convex space, in the sense that one can define linear combinations with non-negative real coefficients adding up to $1$ in a consistent manner (see \cite[2.1]{MR2799809}). Now that convex combinations make sense, so does the notion of extreme point, and \cite[5.2]{MR2799809} argues that a unital morphism $N\to M$ has factorial commutant $N'\cap M$ if and only if the class of the morphism in $\pi_u$ is an extreme point. In conclusion, tbere is a connection between (a) a question by Sorin Popa of whether factors $N$ admitting an embedding into $M$ also admit one with factorial commutant and (b) the geometry of the convex set $\pi_u$. 

Consider now the sheaf $\cat{W}^*_N$ based on $M=R^\omega\otimes\cB(\ell^2)$, the corresponding monoid $\pi$ (no longer just partial; addition is always defined becaue tensoring with $\cB(\ell^2)$ produces enough room to add any two projections), and the enveloping group $K(\pi)$ of the partial obtained by formally adjoining additive inverses to elements of $\pi$. As a continuation of Brown's work, the convex structure on $\pi_u$ is extended in \cite{MR2997389} to a Banach space structure on $K(\pi)$ (more on this circle of ideas in the next section). $\pi_u$ sits inside this Banach space as a closed, convex, bounded subset. One possible approach to proving the existence of extreme points, then, would be to show that $K(\pi)$ is a dual Banach space; $\pi_u$ would then be compact in the weak$^*$ topology, and it would have plenty of extreme point by the Krein-Milman theorem. 

\Cref{th.main} can then be viewed as pointing roughly in the correct direction on the predual front. It implies that $K(\pi)$, with all of the structure just mentioned, is a Dedekind complete Banach lattice. This has some (admittedly flimsy) connection to being a dual, as the following result illustrates (originally due to Nakano; \cite[1.7]{MR1873025} or \cite[2.1.4]{MR1128093}):

\begin{proposition}
The Banach lattice $\cC_\bR(K)$ of real-valued continuous functions on a compact Hausdorff space $K$ is Dedekind complete in the sense of \Cref{def.ded} if and only if $K$ is Stonean, i.e. closures of open sets are open. \qedhere
\end{proposition}

The reason why this is relevant is that being Stonean is only slightly weaker than $\cC(K)$ being a dual Banach space. Indeed, this latter condition is equivalent to the algebra $\cC_\bC(K)$ of \define{complex}-valued continuous functions on $K$ being $W^*$ (this is more or less \cite[III.3.5]{MR1873025}), which in turn is equivalent to $K$ being \define{hyper}stonean. This means that the space is stonean and in some sense has ``enough measures'' (\cite[1.14, 1.18]{MR1873025}). 

As acknowledged above, this is not strong evidence for $K(\pi)$ being a dual Banach space, but these ideas do seem to somehow fit together quite naturally, and the order-theoretic aspect of invariants such as $\pi$ complement the metric side of things (Banach space structures, etc.) very nicely.

\section{Divisibility and cone structures}\label{se.cone}

In some sense, the main theme in \cite{MR2997389} is the promotion of the ``discrete'' addition $+$ on a partial monoid of the form $\pi(\cat{W}^*_N)$ (notation as in \Cref{ex.Wast_mor}) to a ``continuous addition'' (roughly speaking an action of the semiring $(\bR_+,+,\cdot)$ of non-negative reals with the usual addition and multiplication) under certain conditions on the $W^*$-algebras $M$ and $N$. In a slightly less accurate way, this is also true of \cite{MR2799809}, where invariants analogous to $\pi(\cat{W}^*_N)$ constructed using only unital morphisms are endowed with convex structures (i.e. linear combinations with non-negative real coefficients adding up to $1$ are made sense of). We take up this same theme in the present section. 

First, we fix some terminology:

\begin{definition}\label{def.cone}
Let $(G,+,0)$ be a partial (abelian) monoid as in \Cref{def.part_mon}. A \define{partial cone structure} on $G$ is an operation $[0,1]\times G\to G$, $(r,g)\mapsto rg$ such that for $g,h\in G$ and $s,t\in[0,1]$:
	\renewcommand{\labelenumi}{(\alph{enumi})}
	\begin{enumerate}
		\item $0g=0\in G$;
		\item $1g=g$;
		\item $s(g+h) = sg+sh$ when either of the two sides is defined;
		\item $(s+t)g=sg+tg$ when $s+t\le 1$;
		\item $(st)g=s(tg)$.
	\end{enumerate} 
\end{definition}

\begin{remark}\label{rem.01}
This is the kind of structure considered in \cite[Appendix]{MR2997389}. A cone is more appropriately thought of as an action by the semiring $(\bR_+,+,\cdot)$ of non-negative integers rather than just $[0,1]$, hence the adjective `partial'. Nevertheless, since all cone structures considered here are of the kind introduced in \Cref{def.cone}, we sometimes drop that adjective.  
\end{remark}

Since cones allow for ``division'' by positive reals $\ge 1$ (and in particular by positive integers), the following term seems appropriate:

\begin{definition}\label{def.div}
An element $g$ of a partial abelian monoid $G$ is \define{divisible} if there exists an element $h\in G$ such that $h+h=g$. The monoid $G$ is divisible if all of its elements are. 
\end{definition}

\begin{remark}
A more precise term in the spirit of \cite{MR2846474} (a paper that will be relevant below, when we specialize to $G=\pi(\cF)$) would be `2-divisible'; $n$-divisibility would substitute $h+\ldots+h$ ($n$ times) for $h+h$ in \Cref{def.div}. There will never be any occasion to use the more elaborate notion, hence the simpler term `divisible'. 
\end{remark}

Let us return to the setup from before, with $M$ being a von Neumann algebra and $\cF$ a sheaf on it.

\begin{definition}\label{def.div_sheaf}
The sheaf $\cF$ is divisible if the monoid $\pi(\cF)$ with its partial monoid structure from \Cref{prop.monoid} is divisible in the sense of \Cref{def.div}. 
\end{definition}

\begin{remark}\label{rem.div_sheaf}
Unpacking the definition somewhat, $\cF$ is divisible if and only if for every projection $p\in M$, every element $x\in\cF(p)$ is in the image of $\cF(\iota):\cF(A)\to\cF(p)$ for some unital inclusion $\iota:\bC p\to A\cong M_2$.  
\end{remark}

As observed above, putting a cone structure on a partial monoid requires divisibility as a bare minimum. In the specific examples we are analyzing though, it turns out this is also enough:

\begin{theorem}\label{th.main_bis}
Let $M$ be a $W^*$-algebra, and $\cF$ a sheaf on it. The monoid $\pi=\pi(\cF)$ admits a cone structure if and only if $\cF$ is divisible. Moreover, when a cone structure exists, it is unique. 
\end{theorem}

We prove a few auxiliary results first.

\begin{lemma}\label{lem.half}
For an element $x\in\pi=\pi(\cF)$, there is at most one $y\in\pi$ such that $y+y=x$.  
\end{lemma}
\begin{proof}
This follows rather quickly from the results of the previous section. Let $y,z\in\pi$ both satisfy the condition: \[ y+y = x = z+z. \] Subtracting $y\wedge z$ from each of $y$ and $z$, and $y\wedge z+y\wedge z$ from $x$ (which makes perfect sense by cancellability), we may as well assume $y\wedge z=0$. But then \Cref{prop.wedge_vee} says that $y\vee z=y+z$, and hence $x$, which dominates both $y$ and $z$, must also dominate their sum: $x\ge y+z$. This can be recast as either $y+y\ge y+z$ or $z+z\ge y+z$. The former implies (via cancellation) that $y\ge z$, while the latter shows the opposite inequality. Since `$\ge$' is a partial order, we are done: $y=z$.    
\end{proof}

The following simple observation will be handy to cite directly, so we record it. First, in order to make sense of the statement, we point out that (by \Cref{lem.half}) expressions such as $\frac x{2^n}$ are unambiguous for $x\in\pi(\cF)$ and positive integers $n$ whenever the sheaf $\cF$ is divisible.  

\begin{lemma}\label{lem.glb_of_dyadics}
Let $\cF$ be a divisible sheaf on a $W^*$-algebra, and $x\in\pi(\cF)$ an element. The greatest lower bound of the set of all $\frac x{2^n}$ for $n$ ranging over the positive integers is $0\in\pi$.  
\end{lemma}
\begin{proof}
The greatest lower bound ($y$, say) exists by \Cref{th.main}. Since \[ y+y\le \frac x{2^n}+\frac x{2^n} = \frac x{2^{n-1}} \] for all positive integers $n$, it follows that $y+y\le y$. This implies $y=0$, as desired. 
\end{proof}

We next take on the uniqueness part of \Cref{th.main_bis}.

\begin{lemma}\label{lem.unique_cone}
In the setting of \Cref{th.main_bis}, there is at most one cone structure on $\pi=\pi(\cF)$. 
\end{lemma} 
\begin{proof}
We know from \Cref{lem.half} and condition (d) of \Cref{def.cone} that multiplication by $\frac 12$ is uniquely determined, and hence, by repreated applications of the same argument, multiplication by any dyadic rational number (i.e. with a power-of-two denominator) is unique. We will be done if we manage to show that in any cone structure, any multiple $sx$ for $x\in\pi$ and $s\in(0,1)$ is the supremum of $tx$ over all dyadic rationals $t\le s$. We focus on this claim for the remainder of the proof. 

Denote by $y$ the supremum in the claim, so that $sx\ge y$. Assume $s$ is not a dyadic rational, or there is nothing to prove. For any positive integer $n$, we can find a dyadic rational $t$ such that $\displaystyle s\in \left(t,t+\frac 1{2^n} \right)$. This makes difference $sx-y$ at most \[ \left(t+\frac 1{2^n}\right)x-tx = \frac x{2^n}. \] Since this holds for all positive integers $n$, the conclusion follows by \Cref{lem.glb_of_dyadics}.
\end{proof}

\begin{proof of main_bis}
One implication (cone $\Rightarrow$ divisible) is clear, and uniqueness is taken care of by \Cref{lem.unique_cone}. We focus on the other implication of the first statement, and so start out with a divisible sheaf $\cF$ on $M$.  

Divisibility allows us to divide and $x\in\pi$ by powers of $2$ and then, by adding such terms of the form $\frac x{2^n}$, to define multiplication by dyadic rationals. Properties (a) - (e) are easily checked for dyadic rational coefficients using only \Cref{lem.unique_cone}. For non-dyadic coefficients $s$, set
	\[
		sx = \sup_{t\le s} tx = \inf_{t\ge s} tx,
	\]
both ranging over dyadics. That the supremum and the infimum coincide is essentially shown in the course of the proof of \Cref{lem.unique_cone} (the argument is the same: One places $s$ inside a small dyadic interval of length $\frac 1{2^n}$, etc.). Checking the properties is now rather routine, and we illustrate by proving only (d) of \Cref{def.cone}.

As in the definition, let $s,t\in(0,1)$ satisfy $s+t\le 1$, and $x\in\pi$ arbitrary. Assume for simplicity that $s+t$ is strictly less than $1$. Placing both $s$ and $t$ inside tiny dyadic intervals $(s-\varepsilon_1,s+\varepsilon_2)$ and $(t-\delta_1,t+\delta_2)$, note
	\begin{multline*}
		(t-\delta_1-\varepsilon_1-\varepsilon_2)x\ =\ (s+t-\varepsilon_1-\delta_1) x - (s+\varepsilon_2) x\ \le \\ \le\ (s+t)x - sx\ \le\  (s+t+\varepsilon_2+\delta_2) x - (s-\varepsilon_1) x\ =\ (t+\delta_2+\varepsilon_1+\varepsilon_2) x. 
	\end{multline*}
Both sides are dyadic multiples of $x$, and making the $\varepsilon$'s and $\delta$'s approach zero squeezes both sides to $tx$. A similar argument works for $s+t=1$, approximating only $s$ by dyadics close to it. 
\end{proof of main_bis}

Let's again specialize to sheaves of the form $\cat{W}^*_N$ for separable $\2$ factors $N$. The divisibility hypothesis of \Cref{th.main_bis} would be satisfied if, for instance, we knew that for any projection $p\in M$, the relative commutant of any separable von Neumann subalgebra of $pMp$ contains $M_2$ unitally (in which case we say that $M$ has property $\dv$). Indeed, applying property $\dv$ to images of $N$ through not-necessarily-unital morphisms $N\to M$, we recover the characterization of divisibility for $\cat{W}^*_N$ noted in \Cref{rem.div_sheaf}.  

When $M$ is an ultrapower $X^\omega$ of a separable $\2$ factor $X$, $\dv$ is in fact equivalent to $X$ being McDuff (i.e. $X\otimes R\cong X$). The direction McDuff $\Rightarrow$ $\dv$ follows from \cite[Lemma 17]{MR2997389} (see also \cite[3.2.3]{MR2799809}, which is a precursor for unital morphisms when $X=R$, but contains the main ideas of the proof). For the opposite implication, it is known that McDuffness follows from the non-commutativity of the diagonal morphism $X\to X^\omega$; this is the reformulation given in \cite[2.7]{MR2846474} to one of the main results of \cite{MR0281018}. This same discussion applies to $M=X^\omega\otimes\cB(\ell^2)$, which is the choice relevant to \cite{MR2997389}. 

In any event, this particular case of \Cref{th.main_bis} recovers the cone structure on $\pi(\cF)$ from a somewhat different perspective: We still need to check divisibility (by $2$, in Sherman's more refined terminology \cite{MR2846474}), but the rest of the structure and the cone conditions in \Cref{def.cone} follow mechanically and uniquely, and for very general reasons that are not specific to being McDuff or having a large fundamental group.

\section{Completely positive / contractive maps as homomorphisms}\label{se.adjoints}

As outlined briefly in the introduction, the aim of this short section is to explain why \Cref{ex.Wast_mor,ex.cp,ex.cc} are actually the same kind of sheaf (\Cref{cor.adjoints}). There are no proofs in this section; although not difficult, they would take us too far afield, and will appear elsewhere \cite{C}. 

Let us focus on \Cref{ex.cp}, to fix ideas. The claim is that any von Neumann algebra $N$ comes equipped with a canonical completely positive unital (and as always, normal) map $\varphi_N: N\to\wt N$ such that any completely positive unital map $N\to M$ factors uniquely as
	\[
		\tikz[anchor=base]{
  		\path (0,0) node (1) {$N$} +(4,0) node (3) {$M$} +(2,-1) node (2) {$\wt N$};
  		\draw[->] (1) -- (3);
  		\draw[->] (1) -- (2) node[pos=.5,auto,swap] {$\scriptstyle \varphi_N$};
  		\draw[->] (2) -- (3); 
 		}
	\]
for a $W^*$ morphism $\wt N\to M$. The uniqueness part of the assertion then makes the whole thing functorial: $N\mapsto\wt N$ is a functor from the category ($W^*$-algebras, completely positive unital maps) to ($W^*$-algebras, $W^*$ morphisms), and in fact the universality property of $\varphi_N$ says that this functor is left adjoint to the inclusion of categories going in the opposite direction (a $W^*$ morphism being completely positive). The analogous discussion applies to completely contractive maps. 

Denoting by $\cat{W}^*$, $\cat{CPU}$, and $\cat{CCU}$ the categories of $W^*$-agebras and unital $W^*$ morphisms, completely positive maps, and completely contractive maps respectively, the statement is:

\begin{proposition}\label{prop.adjoints}
The inclusion functors $\cat{W}^*\to\cat{CPU}$ and $\cat{W}^*\to\cat{CCU}$ both have left adjoints. \qedhere
\end{proposition}

Denote by $N\mapsto\wh N$ the left adjoint to the second inclusion $\cat{W}^*\to \cat{CCU}$ of the statement. Consider the bijection $\cat{CPU}(\wt N\to M)\cong \cat{W}^*(N\to M)$ induced by $\varphi_N$, and similarly for $\cat{CCU}$. Applying these bijections to all corner subalgebras $pMp$ instead of just $M$, we get:

\begin{corollary}\label{cor.adjoints}
For any von Neumann algebra $M$, the canonical completely positive map $N\to\wt N$ induces an isomorphism $\cat{CP}_N\cong\cat{W}^*_{\wt N}$ of sheaves on $M$. Similarly, we have a canonical isomorphism $\cat{CC}_N\cong\cat{W}^*_{\wh N}$. \qedhere
\end{corollary}

It seems unlikely that \define{all} sheaves as in \Cref{def.sheaf} are of the form $\cat{W}^*_N$, but I do not know any counterexamples. Other natural examples, such as the sheaf associating to $A\in\cC_M$ the set of unitaries or contractions in the relative commutant $A'\cap 1_AM1_A$, are even more obviously of that form than \Cref{ex.cp,ex.cc}.



\end{document}